\documentclass{article}
\usepackage{arxiv}
\usepackage{amsmath}
\usepackage{amssymb}
\usepackage{amsthm}

\theoremstyle{plain}
\newtheorem{theorem}{Theorem}[section]

\newtheorem{corollary}{Corollary}[section]
\theoremstyle{definition}
\newtheorem{definition}{Definition}[section]

\newtheorem{example}{Example}[section]

\begin{document}
	
	\title{Study of the generalized von mangoldt function defined by L-additive function}
	
	\author{
		{ \sc Es-said En-naoui } \\ 
		Faculty Of Science And Technics,University Sultan Moulay Slimane\\ Morocco\\
		essaidennaoui1@gmail.com\\
		\\
	}
	
	\maketitle

	\begin{abstract}
		The main object of this paper is to study the generalized von mangoldt function using the L-additive function, which can help us give many result about the classical arithmetic function.
	\end{abstract}
	
	\section{Introduction}	
	In this work, we address two major result: Dirichlet product of generalized von Mangoldt function with an arithmetic function $f$ and dirichlet serie of many arithmetics function completely additive .
	\medskip
	
	In this article, we study the generalized function von Mangoldt by the L-additive function to find alternative proof for many series expansions that depend on the arithmetic function additive and completely additive.
	The methods readily generalize, and can be applied to other L-additive functions. Our  principal result are that :
	$$
	\sum \limits_{n\geq 1} \frac{\Lambda_{f}(n)}{n^s}=
	\sum \limits_{p}\frac{f(p)}{h_f(p)p^{s}-h_f(p)}
	$$	where $f$ is L-additive function  with $h_f$ is nonzero-valued .
	\medskip
		
	First of all, to cultivate analytic number theory one must acquire a considerable skill for operating with arithmetic functions. we begin with a few elementary considerations. 
	\medskip
	
	\begin{definition}[arithmetic function]
		An \textbf{arithmetic function} is a function $f:\mathbb{N}\longrightarrow \mathbb{C}$ with
		domain of definition the set of natural numbers $\mathbb{N}$ and range a subset of the set of complex numbers $\mathbb{C}$.
	\end{definition}
	
	\begin{definition}[multiplicative function]
		A function $f$ is called an \textbf{multiplicative function} if and
		only if : 
		\begin{equation}\label{equ-1}
			f(nm)=f(n)f(m)
		\end{equation}
	
		for every pair of coprime integers $n$,$m$. In case (\ref{equ-1}) is satisfied for every pair of integers $n$ and $m$ , which are not necessarily coprime, then the function $f$ is
		called \textbf{completely multiplicative}.
	\end{definition}
	
	Clearly , if $f$ are a multicative function , then
	$f(n)=f(p_1^{\alpha _1})\ldots f(p_s^{\alpha _s})$, 
	for any positive integer $n$ such that 
	$n = p_1^{\alpha _1}\ldots  p_s^{\alpha _s}$ , and if $f$ is completely multiplicative , so we have : $f(n)=f(p_1)^{\alpha _1}\ldots f(p_s)^{\alpha _s}$.
	
	The functions defined above are widely studied in the literature, (see, e.g., \cite{KM,Mc,Sc, Sh}). 
	
	\bigskip
	\begin{definition}[additive function]
		A function $f$ is called an \textbf{additive function} if and
		only if : 
		\begin{equation}
			f(nm)=f(n)+f(m)
		\end{equation}
		for every pair of coprime integers $n$,$m$. In case (3) is satisfied for every pair of integers $n$ and $m$ , which are not necessarily coprime, then the function $f$ is
		called \textbf{completely additive}.
	\end{definition}
	Clearly , if $f$ are a additive function , then
	$f(n)=f(p_1^{\alpha _1})+\ldots +f(p_s^{\alpha _s})$, 
	for any positive integer $n$ such that 
	$n = p_1^{\alpha _1}\ldots  p_s^{\alpha _s}$ , and if $f$ is completely additive , so we have : $f(n)=\alpha _1f(p_1)+\ldots +\alpha _sf(p_s)$.
	
\begin{example}
	This is the some classical arithmetic functions used in this paper :
	\begin{enumerate}
		\item \textbf{the arithmetic logarithmic derivative } :\;$Ld(n)=\sum \limits_{p^{\alpha}\parallel n}\frac{\alpha}{p}$.
		\item \textbf{The number of prime factors of n counting multiplicity } 
		:\;$\Omega(n)=\sum \limits_{p^{\alpha}\parallel n}\alpha$.
		\item \textbf{The generalization of  the number of prime factors of n counting multiplicity}  
		:\;$\Omega_k(n)=\sum \limits_{p^{\alpha}\parallel n}\alpha^k$ .
		\item  The function defined by : $A(n)=\sum \limits_{p^{\alpha}\parallel n}\alpha p$.
		
		\item \textbf{Unit function }: The function defined by 
		$e(n)=\left\{ \begin{array}{cl}
			1 & \textrm{if }\;\;n=1\\
			0& \textrm{for }\;\; n\geq 2\\
		\end{array}\right.$
		\item \textbf{the unit function }: The function defined by $1(n)=1$ for all $n>1$.  
		\item \textbf{Logarithm function }: the (natural) logarithm, restricted to $\mathbb{N}$ and regarded as an arithmetic function.
		\item \textbf{The sum k-th powers of the prime factors of n} :\;$\beta_k(n)=\sum\limits_{p|n}p^k$ .
		\item \textbf{The number of distinct prime divisors of n} :\;$\omega(n)=\beta_0(n)=\sum \limits_{p|n}1$ .
		\item \textbf{The sum of the prime factors of n} :\;$\beta(n)=\beta_1(n)=\sum \limits_{p|n}p$.
		\item \textbf{The Mobiuse function } 
		:\;$\mu(n)=\left\{ \begin{array}{cl}
			1 & \textrm{if }\;\;n=1\\
			0& \textrm{if }\;\; p^2|n\;\; for\; some\; prime\; p\\
			(-1)^{\omega(n)} & \textrm{otherwise}\\
		\end{array}\right.$
		\item \textbf{number of positive divisors of $n$} defined by :  $\tau(n)=\sum \limits_{d|n}1$ .
	\end{enumerate}
	
\end{example}
	\bigskip
	\begin{definition}[L-additive function]
		We say that an arithmetic function $f$ is {\em Leibniz-additive} (or, {\em L-additive}, in short) (see, e.g., \cite{Ufna}) if there is a completely multiplicative function $h_f$ such that 
		\begin{equation}\label{gca}
			f(mn)=f(m)h_f(n)+f(n)h_f(m)
		\end{equation}
		for all positive integers $m$ and $n$. 
	\end{definition}
	Then $f(1)=0$ since $h_f(1)=1$. 
	The property \eqref{gca} may be considered a generalized Leibniz rule.
	For example, the arithmetic derivative $\delta$ is L-additive with $h_\delta(n)=n$, 
	since it satisfies the usual Leibniz rule 
	$$
	\delta(mn)=n\delta(m)+m\delta(n)
	$$
	for all positive integers $m$ and $n$, and the function $h_\delta(n)=n$ is completely multiplicative. 
	Further, all completely additive functions $f$ are L-additive with $h_f(n)=1$. For example, the logarithmic derivative of $n$ is completely additive since
	$$
	{\rm ld}(mn) = {\rm ld}(m)+{\rm ld}(n).
	$$
	
	\begin{theorem}\label{the-1-1}
	Let $f$ be an arithmetic function. If $f$ is L-additive and
	$h_f$ is nonzero-valued, then $f/h_f$ is completely additive.
\end{theorem}

\begin{proof} 
	(see, e.g., \cite[Theorem 2.1]{Ufna})	
\end{proof}

\begin{theorem}\label{totally}
	Let $f$ an arithmetic function, if $n=\prod_{i=1}^{s}p_{i}^{\alpha_{i}}$ is the prime factorization of $n$ and $f$ is L-additive with $h_f(p_1),\dots,h_f(p_s)\ne 0$, then 
	$$
	f(n)=h_f(n)\sum_{i=1}^s\frac{\alpha_i f(p_i)}{h_f(p_i)}. 
	$$		
\end{theorem}
\begin{proof}
	(see, e.g., \cite[Theorem 2.4]{Ufna})
\end{proof}
\begin{corollary}\label{cor-1-1}
	Let $f$ an arithmetic function, if $f$ is L-additive  with $h_f$ is nonzero-valued, then we have :
	$$
	(f\ast h_f)(n) =\frac{1}{2}f(n)\tau(n)
	$$		
\end{corollary}
\begin{proof}
	(see, e.g., \cite[Corollary 3.1]{Ufna})
\end{proof}
\medskip

Further, all completely additive functions $f$ are L-additive with $h_f(n) = 1$ , then the extended of completely addtive function to the set of rational number $\mathbb{Q}$ give us this formula :
$$f\left(\frac{n}{m} \right) =f(n)-f(m)$$ 

For example, the logarithmic derivative of $n$ is completely additive, then we have :
$$ Ld\left(\frac{n}{m} \right)=Ld(n)-Ld(m)$$
	
	Let $f$ and $g$ be arithmetic functions. Their {\it Dirichlet convolution} is :
	\begin{equation}\label{dir-prod}
	(f\ast g)(n)=\sum_{\substack{a,b=1\\ab=n}}^n f(a)g(b)=\sum_{\substack{d|n}}^n f(d)g\left(\frac{n}{d} \right) .
	\end{equation}

	where the sum extends over all positive divisors $d$ of $n$ , or equivalently over all distinct pairs $(a, b)$ of positive integers whose product is $n$.\\
	In particular, we have $(f*g)(1)=f(1)g(1)$ ,$(f*g)(p)=f(1)g(p)+f(p)g(1)$ for any prime $p$ and for any power prime $p^m$ we have :
	\begin{equation}
		(f*g)(p^m)=\sum \limits_{j=0}^m f(p^j)g(p^{m-j})
	\end{equation}


	In this paper, we consider L-additive functions, especially from the viewpoint that they are a way to generalizations of von mangoldt function. In the next section, we present their general basic properties. In the last section, we study the application of this generalizations  in terms of the Dirichlet convolution and Dirichlet series .
	
\section{The generalized von mangoldt function using L-additive function :}
	In this section , let $f$ L-additive function with $h_f$ is nonzero-valued, then  Now we defined the von Mangoldt function  related to The function $f$ by :
	\begin{equation}\label{def-mangol}
		\Lambda_f (n)={\begin{cases}\frac{f(p)}{h_f(p)}&{\text{if }}n=p^{k}{\text{ for some prime }}p{\text{ and integer }}k\geq 1,\\0&{\text{otherwise.}}\end{cases}}
	\end{equation}
	then we have this result :
	\begin{theorem}\label{the-2-1}
		If $n\geq1$ then we have : 
		$$
		f(n)=h_f(n)\sum_{\substack{d|n}} \Lambda_f(d)
		$$
		That mean by using dirichlet convolution : $f=h_f\ast h_f\Lambda_f$
	\end{theorem}
	\begin{proof}
		If $n = p_1^{\alpha _1}\ldots  p_s^{\alpha _s}$ then we have :
		$$
		\sum_{\substack{d|n}} \Lambda_f(d)=\sum_{\substack{i=1}}^s \sum_{\substack{k=1}}^i \Lambda_f(p_i^k)=\sum_{\substack{i=1}}^s \sum_{\substack{k=1}}^i \frac{f(p_i)}{h_f(p_i)}=
		\sum_{\substack{i=1}}^s \frac{if(p_i)}{h_f(p_i)}=\frac{f(n)}{h_f(n)}
		$$
		as claimed \\
	\end{proof}
	\begin{theorem}\label{theo-2-2}
		for every positive integer $n$ we have : 
		$$
		{\displaystyle \Lambda_f (n)
			=\sum _{d\mid n}\frac{\mu \left(\frac{n}{d} \right)f\left(d \right) }{h_f\left(d\right)}}
		=-\sum_{\substack{d|n}}\frac{\mu(d)f(d)}{h_f(d)}
		$$
		That is we have : $$\Lambda_f=\mu\ast \frac{f}{h_f}=-1\ast\frac{\mu f}{h_f}$$
	\end{theorem}
	\begin{proof}
		By Theorem (\ref{the-2-1}) applying Mobius inversion, we have $\Lambda_f=\mu\ast \frac{f}{h_f}$ , Note that :
		\begin{align*}
			\sum _{d\mid n}\frac{\mu \left(\frac{n}{d} \right)f\left(d \right) }{h_f\left(d\right)} &=
			\sum _{d\mid n} \mu(d)\frac{f\left(\frac{n}{d} \right)}{h_f\left(\frac{n}{d} \right)}
			=\sum \limits_{d|n} \mu(d)\frac{h_f(d)}{h_f(n)}
			\bigg( \frac{h_f(d)f(n)-h_f(n)f(d)}{h_f^2(d)}\bigg)
			\\
			&=\sum \limits_{d|n} \frac{\mu(d)h^2_f(d)f(n)}{h_f(n)h^2_f(d)}
			-\frac{\mu(d)h_f(n)h_f(d)f(d)}{h_f(n)h^2_f(d)}
			\\
			&=\frac{f(n)}{h_f(n)}\sum \limits_{d|n} \mu(d)
			-\sum \limits_{d|n}\frac{\mu(d)f(d)}{h_f(d)}
			\\
			& = \frac{f(n)e(n)}{h_f(n)}-\sum \limits_{d|n}\frac{\mu(d)f(d)}{h_f(d)}
			\\
			&=-\sum \limits_{d|n}\frac{\mu(d)f(d)}{h_f(d)}
		\end{align*}
		Hence $\Lambda_f (n)=-1\ast\frac{\mu f}{h_f}$. This completes the proof .
	\end{proof}
	
	\begin{corollary}\label{cor-2-1}
		Let $f$ an arithmetic function. If $f$ is L-additive and $h_f$ is nonzero-valued, then :
		\begin{eqnarray}
			(\tau\ast \Lambda_f)(n) =\frac{f(n)\tau(n)}{2h_f(n)} 
		\end{eqnarray}
	\end{corollary}
	\begin{proof}
		By the Corollary  (\ref{cor-1-1}) we have that  :
		$$
		(f\ast h_f)(n) =\frac{1}{2}f(n)\tau(n)
		$$
		Since from the theorem (\ref{the-2-1}) we know :
		$$
		f(n)=\left( h_f\ast h_f\Lambda_f\right) (n)
		$$
		Then we find that :
		$$
		\left(h_f\ast f \right)(n)= \left(h_f\ast  h_f\ast h_f\Lambda_f \right)(n)=
		h_f(n)\left(1\ast 1\ast\Lambda_f \right)(n) 
		$$
		We conclude that :
		$$
		h_f(n)\left(\tau\ast\Lambda_f \right)(n) =\frac{1}{2}\tau(n)f(n)
		$$
		This completes the proof of Theorem
	\end{proof}
\begin{theorem}\label{theo-2-3}
	Let $g$ an arithmetic function , if $g$ is completely additive then we have :
	\begin{equation}
		\left(1\ast g\Lambda_f \right) (n)=\frac{1}{2}\sum \limits_{p^{\alpha}||n} \frac{\alpha(\alpha+1)f(p)g(p)}{h_f(p)}
	\end{equation}
\begin{proof}
Let $g$ an arithmetic function completely additive , then  :
\begin{align*}
\left(1\ast g\Lambda_f \right) (n)&=	\sum _{d\mid n} g(d)\Lambda_f(d) =\sum \limits_{p^{\alpha}||n}\sum \limits_{i=1}^{i=\alpha} g(p^{i})\Lambda_f(p^{i}) 
	\\
	&=\sum \limits_{p^{\alpha}||n}\sum \limits_{i=1}^{i=\alpha} \frac{ig(p)f(p)}{h_f(p)}
	\\
	&=\sum \limits_{p^{\alpha}||n}\frac{g(p)f(p)}{h_f(p)}\sum \limits_{i=1}^{i=\alpha} i
	\\
	&=\sum \limits_{p^{\alpha}||n} \frac{\alpha(\alpha+1)f(p)g(p)}{2h_f(p)}
\end{align*}	
\end{proof}
\end{theorem}
\begin{theorem}\label{theo-2-4}
	Let $g$ an arithmetic function , if $g$ is completely additive then we have :
	\begin{equation}
		\left(\Lambda_f\ast g \right) (n)=\frac{f(n)g(n)}{h_f(n)}-\left(1\ast g\Lambda_f \right) (n)
	\end{equation}
\end{theorem}
\begin{proof}
	Let $g$ an arithmetic function completely additive ,since $g\left(\frac{n}{d} \right)=g(n)-g(d) $ then  :
	\begin{align*}
		\left(\Lambda_f\ast g \right) (n)&=	\sum _{d\mid n}\Lambda_f(d)g\left(\frac{n}{d} \right)  = \sum _{d\mid n}\Lambda_f(d)\left(g(n)-g(d) \right) 
		\\
		&=\sum _{d\mid n}\Lambda_f(d)g(n)-\sum _{d\mid n}\Lambda_f(d)g(d)
		\\
		&=g(n)\sum _{d\mid n}\Lambda_f(d)-\sum _{d\mid n}\Lambda_f(d)g(d)
	\end{align*}
Since :
$$
\sum _{d\mid n}\Lambda_f(d)g(d)=\left( 1\ast g\Lambda_f\right) (n)
$$	
And by theorem (\ref{the-2-1}) we have :
$$
\sum _{d\mid n}\Lambda_f(d)=\frac{f(n)}{h_f(n)}
$$
Then we conclude that :
$$
\left(\Lambda_f\ast g \right) (n)=\frac{f(n)g(n)}{h_f(n)}-\left( 1\ast g\Lambda_f\right) (n)
$$
\end{proof}
\begin{theorem}\label{theo-2-5}
	Let $f$ and $g$ two L-additive functions , then we have :
\begin{equation}
	\left(\Lambda_f\ast\Lambda_g \right) (n)={\begin{cases}\frac{(\alpha-1)f(p)g(p)}{h_f(p)h_g(p)}&{\text{if }}n=p^{\alpha}{\text{ for some prime }}p{\text{ and integer }}\alpha\geq 1,\\
		\frac{f(p)g(q)}{h_f(p)h_g(q)}+\frac{f(q)g(p)}{h_f(q)h_g(p)}
			&{\text{if }}n=p^{\alpha}q^{\beta}{\text{ for some prime }}p,q{\text{ and integer }}\alpha,\beta\geq 1
			\\0&{\text{otherwise.}}\end{cases}}
\end{equation}
\end{theorem}
\begin{proof}
Let $f$ and $g$ two L-additive functions , and $n>1$ .\\
if $n=p^{\alpha}$ then we have :
\begin{align*}
	\left(\Lambda_f\ast\Lambda_g \right) (n)&=	\sum _{d\mid n}\Lambda_f(d)\Lambda_g\left(\frac{n}{d} \right)  = \sum \limits_{i=1}^{\alpha-1} \Lambda_f(p^{i})\Lambda_g(p^{\alpha-i}) 
	\\
	&= \sum \limits_{i=1}^{\alpha-1} \frac{f(p)}{h_f(p)}\frac{g(p)}{h_g(p)}
	\\
	&= \frac{(\alpha-1)f(p)g(p)}{h_f(p)h_g(p)}
\end{align*}
if $n=p^{\alpha}q^{\beta}$ then :
\begin{align*}
	\left(\Lambda_f\ast\Lambda_g \right) (n)&=	\sum _{d\mid n}\Lambda_f(d)\Lambda_g\left(\frac{n}{d} \right)  = \sum \limits_{i=1}^{\alpha} \Lambda_f(p^{i})\Lambda_g\left(\frac{n}{p^{i}} \right) +\sum \limits_{i=1}^{\beta} \Lambda_f(q^{i})\Lambda_g\left(\frac{n}{q^{i}} \right)
	\\
	&= \sum \limits_{i=1}^{\alpha} \Lambda_f(p^{i})\Lambda_g\left(p^{\alpha-i}q^{\beta} \right) +\sum \limits_{j=1}^{\beta} \Lambda_f(q^{j})\Lambda_g\left(p^{\alpha} q^{\beta-j}\right)
	\\
	&= \Lambda_f(p^{\alpha})\Lambda_g(q^{\beta})+\Lambda_f(q^{\beta})\Lambda_g(q^{\alpha})
	\\
	&= \frac{f(p)g(q)}{h_f(p)h_g(q)}+\frac{f(q)g(p)}{h_f(q)h_g(p)}
\end{align*}
Now if $\omega(n)>2$ , the for every divisor $d$ of $n$ we have $\Lambda_f(d)=0$ or $\Lambda_f\left(\frac{n}{d} \right) $ , then $\left(\Lambda_f\ast\Lambda_g \right) (n)=0$
\end{proof}
	\begin{theorem}\label{theo-2-6}
	Let $s$ a complex number  , if  then we have :
	$$
	\sum \limits_{n\geq 1} \frac{\Lambda_{f}(n)}{n^s}=
	\sum \limits_{p}\frac{f(p)}{h_f(p)p^{s}-h_f(p)}
	$$	
\end{theorem}

\begin{proof}
	Let $s$ a number complex such that $Re(s)>0$ , then 
	\begin{align*}
		\sum \limits_{n\geq 1} \frac{\Lambda_{f}(n)}{n^s} &=\frac{\Lambda_{f}(1)}{1^s}+\frac{\Lambda_{f}(2)}{2^s}+\frac{\Lambda_{f}(3)}{3^s}+\frac{\Lambda_{f}(4)}{4^s}+\frac{\Lambda_{f}(5)}{5^s}+\ldots +\frac{\Lambda_{f}(16)}{16^s}+\ldots
		\\
		&=\frac{f(2)}{h_f(2)2^{s}}+\frac{f(3)}{h_f(3)3^{s}}+\frac{f(2)}{h_f(2)2^{2s}}+\frac{f(5)}{h_f(5)5^{s}}+\frac{f(7)}{h_f(7)7^{s}}+\frac{f(2)}{h_f(2)2^{3s}}+\ldots+\frac{f(2)}{h_f(2)2^{4s}}+\ldots
		\\
		&=\sum \limits_{p}\sum \limits_{k\geq 1}\frac{f(p)}{h_f(p)p^{ks}}= \sum \limits_{p}\frac{f(p)}{h_f(p)}\sum \limits_{k\geq 1}\frac{1}{p^{ks}}
		\\
		& =\sum \limits_{p}\frac{f(p)}{h_f(p)}\sum \limits_{k\geq 1}\bigg(\frac{1}{p^{s}} \bigg)^k=
		\sum \limits_{p}\frac{f(p)}{h_f(p)}.\frac{1}{p^s}.\frac{1}{1-\frac{1}{p^s}}
		\\
		&=\sum \limits_{p}\frac{f(p)}{h_f(p)p^{s}-h_f(p)}
	\end{align*}
	which completes the proof
\end{proof}

\subsection{ The  derivatives of arithmetical functions using L-additive function}
Let $f$ L-additive function with $h_f$ is nonzero-valued, Now we defined the derivatives of arithmetical functions  related to The function $f$ by :
\begin{definition}
	For any arithmetical function $g$ we define its derivative $g'$ to be the arithmetical function given by the equation : 
	$$
	g'(n)=\frac{g(n)f(n)}{h_f(n)}\;\;\;\;for \;\;\; n\geqslant 1
	$$
\end{definition}
Since $e(n)\frac{f(n)}{h_f(n)}=0$ for all $n$ we have $e'(n)=0$. \\
Since $1'(n)=\frac{f(n)}{h_f(n)}$ for all $n$ Hence, the formula $\sum_{\substack{d|n}} \Lambda_f(d)=\frac{f(n)}{h_f(n)}$ can be 
written as 
\begin{equation}\label{equ-9}
	1'(n)=\left( 1\ast \Lambda_f \right) (n)
\end{equation}
This concept of derivative using L-additive function shares many of the properties of the ordinary  derivative discussed in elementary calculus. For example, the usual rules for differentiating sums and products also hold if the products are Dirichlet products. 
\begin{theorem}
	If $g$ and $h$ are arithmetical functions we have: 
	\begin{enumerate}
		\item[a)] $\left(g+h \right)'(n)=g'(n)+h'(n) $
		\item[b)] $\left(g\ast h \right)'(n)=\left(g'\ast h \right) (n)+\left(g\ast h' \right) (n) $
	\end{enumerate}
\end{theorem}
\begin{proof}
	The proof of $(a)$ is immediate. Of course, and to prove $(b)$ we use the identity  $\frac{f(n)}{h_f(n)}=\frac{f(d)}{h_f(d)}+\frac{f(\frac{n}{d})}{h_f(\frac{n}{d})}$ to write :
	\begin{align*}
		\left(g\ast h \right)'(n)
		&=\sum \limits_{d|n} g(d)h\left(\frac{n}{d} \right)\frac{f(n)}{h_f(n)}
		\\
		&=\sum \limits_{d|n} g(d)\frac{f(d)}{h_f(d)}h\left(\frac{n}{d} \right)+
		\sum \limits_{d|n} g(d)h\left(\frac{n}{d} \right)\frac{f(\frac{n}{d})}{h_f(\frac{n}{d})}
		\\
		&=\sum \limits_{d|n}  \frac{g(d)f(d)}{h_f(d)}h\left(\frac{n}{d} \right)+
		\sum \limits_{d|n} g(d)\frac{h\left(\frac{n}{d} \right)f(\frac{n}{d})}{h_f(\frac{n}{d})}
		\\
		& =\left(g'\ast h \right) (n)+\left(g\ast h' \right) (n)
	\end{align*}	
\end{proof}
	\begin{theorem}[Ennaoui-Selberg identity.]\label{theo-2-8}
	For $n>1$ we have:
	\begin{equation}
		\frac{\Lambda_f(n)f(n)}{h_f(n)}+\sum \limits_{d|n} \Lambda_f(d)\Lambda_f\left(\frac{n}{d} \right) =\sum \limits_{d|n}\mu\left(\frac{n}{d} \right)\frac{f^2\left(d \right) }{h^2_f\left( d\right) }
	\end{equation}
	by using Dirichlet product that mean :
	\begin{equation}
		\frac{f\Lambda_f}{h_f}+\Lambda_f\ast\Lambda_f=\mu\ast\left(\frac{f}{h_f} \right) ^2
	\end{equation}
\end{theorem}
\begin{proof}
	Equation (\ref{equ-9}) states that $1'=1\ast\Lambda_f$. Differentiation of this equation gives us 
	$$
	1''=1'\ast\Lambda_f+1\ast\Lambda_f'
	$$
	Since $1'=1\ast\Lambda_f $ we have :
	$$
	1''=\left( 1\ast\Lambda_f\right) \ast\Lambda_f+1\ast\Lambda_f'
	$$
	Now we multiply both sides by $\mu=1^{-1}$ to obtain :
	$$
	\mu\ast 1''=\Lambda_f' +\Lambda_f\ast\Lambda_f
	$$
	This is the required identity. 
\end{proof}
\medskip

\subsection{Results : completely additive function}
	As we knows the arithmetic function $f$ completely additive is L-additive function with $h_{f}(n)=1(n)=1$ for every integer not null, then we have the von Mangoldt function  related to The function $f$ defined by :
	\begin{equation}\label{equ-21}
		\Lambda_f (n)={\begin{cases}f(p)&{\text{if }}n=p^{k}{\text{ for some prime }}p{\text{ and integer }}k\geq 1,\\0&{\text{otherwise.}}\end{cases}}
	\end{equation}
Substituting $h_f(n)=1$ into all results in previous section to find that :	
	\begin{corollary}\label{cor-2-2}
		Let $f$ an arithmetic function completely additive, then : 
		$$
		f(n)=\sum_{\substack{d|n}} \Lambda_f(d)
		$$
		That mean by using dirichlet convolution : $f=1\ast \Lambda_f$
	\end{corollary}
	\begin{corollary}\label{cor-2-3}
		Let $f$ an arithmetic function completely additive, then we have: 
		$$
		{\displaystyle \Lambda_f (n)
			=\sum _{d\mid n}\mu \left(\frac{n}{d} \right)f\left(d \right)
			=-\sum_{\substack{d|n}}\mu(d)f(d)}
		$$
		That is we have : $$\Lambda_f=\mu\ast f=-1\ast \mu f$$
	\end{corollary}
\begin{corollary}\label{cor-2-4}
	Let $g$ an arithmetic function , if $g$ is completely additive then we have :
	\begin{equation}
		\left(1\ast g\Lambda_f \right) (n)=\frac{1}{2}\sum \limits_{p^{\alpha}||n} \alpha(\alpha+1)f(p)g(p)
	\end{equation}
\end{corollary}
\begin{corollary}\label{cor-2-5}
	For $n>1$ we have :
	\begin{equation}
		\left(\Lambda_f\ast g \right) (n)=f(n)g(n)-\left(1\ast g\Lambda_f \right) (n)
	\end{equation}
\end{corollary}
\begin{corollary}\label{cor-2-6}
	if $f$ and $g$ is two arithmetic function completely additive then we have :
	\begin{equation}
		\left(\Lambda_f\ast\Lambda_g \right) (n)={\begin{cases}(\alpha-1)f(p)g(p)&{\text{if }}n=p^{\alpha}{\text{ for some prime }}p{\text{ and integer }}\alpha\geq 1,\\
				f(p)g(q)+f(q)g(p)
				&{\text{if }}n=p^{\alpha}q^{\beta}{\text{ for some prime }}p,q{\text{ and integer }}\alpha,\beta\geq 1
				\\0&{\text{otherwise.}}\end{cases}}
	\end{equation}
\end{corollary}
\begin{corollary}[Ennaoui-Selberg identity.]\label{cor-2-7}
	For $n>1$ we have:
	\begin{equation}
		\Lambda_f(n)f(n)+\sum \limits_{d|n} \Lambda_f(d)\Lambda_f\left(\frac{n}{d} \right) =\sum \limits_{d|n}\mu\left(\frac{n}{d} \right) f^2\left(d \right)
	\end{equation}
	by using Dirichlet product we have that :
	\begin{equation}
		f\Lambda_f+\Lambda_f\ast\Lambda_f=\mu\ast f^2
	\end{equation}
\end{corollary}
\begin{corollary}\label{cor-2-8}
	Let $f$ an arithmetic function. If $f$ is completely additive, then :
	\begin{eqnarray}
		(\tau\ast \Lambda_f)(n) =\frac{1}{2}f(n)\tau(n) 
	\end{eqnarray}
\end{corollary}
\begin{corollary}\label{cor-2-9}
	Let $s$ a number complex such that $Re(s)>0$ , then we have :
	$$
	\sum \limits_{n\geq 1} \frac{\Lambda_{f}(n)}{n^s}=
	\sum \limits_{p}\frac{f(p)}{p^{s}-1}
	$$	
\end{corollary}
\section{Application : classical arithmetics function completely additive}
\subsection{The function $\Omega$ of the number of prime factors of n counting multiplicity .}

We know the $\Omega$ function  of the number of prime factors of n counting multiplicity is L-additive function with $h_{\Omega}(n)=1$ then The von mangoldt function of $\Omega$ is defined by :
$$
\Lambda_{\Omega} (n)={\begin{cases} 1 &{\text{if }}n=p^{k}{\text{ for some prime }}p{\text{ and integer }}k\geq 1,\\0&{\text{otherwise.}}\end{cases}}
$$
then by the corollary (\ref{cor-2-2}) and (\ref{cor-2-2}) we have :
\begin{corollary}
	for $n>1$  we have :
	\begin{equation}
		\Omega(n)=\left( 1\ast \Lambda_{\Omega}\right)(n)
	\end{equation}
	\begin{equation}
		\Lambda_{\Omega}(n)=\left( \mu\ast \Omega\right) (n)=-\left( 1\ast \Omega\mu\right)(n) 
	\end{equation}
\end{corollary}
Substituting $f(n)=\Omega(n)$ into the corollary (\ref{cor-2-8}) to conclude that :
\begin{corollary}\label{cor-3-2}
	For every integer $n>1$ we have
	\begin{equation}
		\left(\Lambda_{\Omega}\ast\tau \right)(n)=\frac{1}{2}\tau(n)\Omega(n)
	\end{equation}
\end{corollary}
if $f=g=\Omega$ then by corollary (\ref{cor-2-4}) we have this result :
\begin{equation}
	\left(1\ast \Omega\Lambda_{\Omega} \right) (n)=\frac{1}{2}\left(\Omega_2(n)+\Omega(n) \right) 
\end{equation}
Substituting $f(n)=\Omega(n)$ into the Ennaoui-Selberg identity (\ref{cor-2-7}) to find that :
\begin{equation}
	\Omega(n)\Lambda_\Omega(n)+\left( \Lambda_\Omega\ast\Lambda_\Omega\right) (n)=\left( \mu\ast \Omega^2\right) (n)
\end{equation}
\begin{theorem}\label{the-3-1}	for every integer $n>0$ we have :
	\begin{equation}
		\left( \Lambda_{\Omega}\ast\beta_k\right)(n)=\Omega(n)\beta_k(n)-\beta_k(n)
	\end{equation}
\end{theorem}
\begin{proof}
	for any positive integer $n$ such that  $n = p_1^{\alpha _1}\ldots  p_s^{\alpha _s}$ we have :
	\begin{align*}
		\left( \Lambda_{\Omega}\ast\beta_k\right)(n)&=	\sum _{d\mid n} \Lambda_{\Omega}(d)\beta_k\left(\frac{n}{d} \right)  =\sum \limits_{p^{\alpha}||n}\sum \limits_{i=1}^{i=\alpha} \Lambda_{\Omega}(p^{i})\beta_k\left(\frac{n}{p^{i}} \right) 
		\\
		&=\sum \limits_{p^{\alpha}||n}\sum \limits_{i=1}^{i=\alpha} \beta_k\left(\frac{n}{p^{i}} \right) 
		\\
		&=\sum \limits_{p^{\alpha}||n}\left( -p^k +\sum \limits_{i=1}^{i=\alpha} \beta_k\left(n\right)\right) 
		\\
		&=\sum \limits_{p^{\alpha}||n} -p^k +\alpha\beta_k\left(n\right)
		\\
		&=\beta_k(n)\sum \limits_{p^{\alpha}||n}\alpha-\sum \limits_{p^{\alpha}||n} p^k 
		\\
		&=\beta_k(n)\Omega(n)-\beta_k(n)
	\end{align*}
\end{proof}

\subsection{the arithmetic logarithmic derivative function :}
	Now we can defined the von mangoldt function associed to arithmetic derivative $Ld$ by : 
$$
\Lambda_{Ld} (n)={\begin{cases}\frac{1}{p} &{\text{if }}n=p^{k}{\text{ for some prime }}p{\text{ and integer }}k\geq 1,\\0&{\text{otherwise.}}\end{cases}}
$$

	Substituting $f(n)=Ld(n)$ into the corollary (\ref{cor-2-2}) and (\ref{cor-2-3}) gives :
	\begin{equation}
		Ld(n)=\left( 1\ast \Lambda_{Ld}\right)(n)
	\end{equation}
	And : 
	\begin{equation}
		\Lambda_{Ld}(n)=\left( \mu\ast Ld\right) (n)=-\left( 1\ast Ld\mu\right)(n) 
	\end{equation}
Now substituting $f(n)=Ld(n)$ into the corollary (\ref{cor-2-7}) and (\ref{cor-2-8}) to get :
\begin{corollary}
	For $n>1$ we have 
	\begin{equation}
		\left(\Lambda_{Ld}\ast\tau \right)(n)=\frac{1}{2}\tau(n)Ld(n)
	\end{equation}
	\begin{equation}
		Ld(n)\Lambda_{Ld}(n)+\left(\Lambda_{Ld}\ast\Lambda_{Ld} \right) (n)=\left( \mu\ast Ld^2 \right)(n) 
	\end{equation}
\end{corollary}
\begin{theorem}\label{the-3-2}
	Let $g$ an arithmetic function , if $g$ is completely additive then we have :
	\begin{equation}
		\left(\Lambda_{Ld}\ast g \right) (n)=g(n)Ld(n)-\frac{1}{2}\sum \limits_{p^{\alpha}||n} \frac{\alpha(\alpha+1)g(p)}{p}
	\end{equation}
\end{theorem}
\begin{proof}
	First take $f=Ld$ in the corollary (\ref{cor-2-5}) to find that :
	$$
		\left(\Lambda_{Ld}\ast g \right) (n)=g(n)Ld(n)-\left(1\ast g\Lambda_{Ld} \right) (n)
	$$
	in the same substituting $f=Ld$ into corollary (\ref{cor-2-4}) gives :
	$$
	\left(1\ast g\Lambda_{Ld} \right) (n)=\frac{1}{2}\sum \limits_{p^{\alpha}||n} \alpha(\alpha+1)Ld(p)g(p)
	$$
	Substituting $Ld(p)=\frac{1}{p}$ completes the proof.
\end{proof}
\begin{theorem}\label{the-3-3}
	For $n>1$ and for $k\in\mathbb{Z}$ we have :
	\begin{equation}
		\left(\Lambda_{Ld}\ast\beta_k \right)(n)=Ld(n)\beta_k(n)-\beta_{k-1}(n) 
	\end{equation}
\end{theorem}
\begin{proof}
	for any positive integer $n$ such that  $n = p_1^{\alpha _1}\ldots  p_s^{\alpha _s}$ we have :
	\begin{align*}
		\left( \Lambda_{Ld}\ast\beta_k\right)(n)&=	\sum _{d\mid n} \Lambda_{Ld}(d)\beta_k\left(\frac{n}{d} \right)  =\sum \limits_{p^{\alpha}||n}\sum \limits_{i=1}^{i=\alpha} \Lambda_{Ld}(p^{i})\beta_k\left(\frac{n}{p^{i}} \right) 
		\\
		&=\sum \limits_{p^{\alpha}||n}\sum \limits_{i=1}^{i=\alpha} 
		\frac{\beta_k\left(\frac{n}{p^{i}} \right) }{p}
		\\
		&=\sum \limits_{p^{\alpha}||n}\left( -\frac{p^k}{p} +\frac{1}{p}\sum \limits_{i=1}^{i=\alpha} \beta_k\left(n\right)\right) 
		\\
		&=\sum \limits_{p^{\alpha}||n} -p^{k-1} +\frac{\alpha}{p}\beta_k\left(n\right)
		\\
		&=\beta_k(n)\sum \limits_{p^{\alpha}||n}\frac{\alpha}{p}-\sum \limits_{p^{\alpha}||n} p^{k-1} 
		\\
		&=\beta_k(n)Ld(n)-\beta_{k-1}(n)
	\end{align*}
\end{proof}

\begin{corollary}
	\begin{equation}
		\left(\Lambda_{Ld}\ast\Omega \right)(n)=\Omega(n)Ld(n)-Ld(n)-\sum \limits_{p^{\alpha}||n}\frac{\alpha^2}{p} 
	\end{equation}
\end{corollary}
\subsection{The function $A$ of the sum of all prime factors in the prime factorization : }
The function $A$ (OEIS A001414) which gives the sum of prime factors (with repetition) of a number $n$ is one of the arithmetic functions that studied by K. ALLADI and P. ERDOS  (see, e.g., \cite{erdos}), So in this section we study this function and we give some result about dirichlet product of this function with many classical arithmetic function.\\
Clearly, the funtion $A$ is completely additive due to the uniqueness of the prime factorization of every integer n .then The von mangoldt function of $A$ is defined by :
$$
\Lambda_{A} (n)={\begin{cases} p &{\text{if }}n=p^{k}{\text{ for some prime }}p{\text{ and integer }}k\geq 1,\\0&{\text{otherwise.}}\end{cases}}
$$
then by the corollary (\ref{cor-2-2}) and (\ref{cor-2-2}) we have :
\begin{corollary}
	for $n>1$  we have :
	\begin{equation}
		A(n)=\left( 1\ast \Lambda_{A}\right)(n)
	\end{equation}
	\begin{equation}
		\Lambda_{A}(n)=\left( \mu\ast A\right) (n)=-\left( 1\ast A\mu\right)(n) 
	\end{equation}
\end{corollary}
Substituting $f(n)=A(n)$ into the corollary (\ref{cor-2-8}) gives that :
\begin{corollary}\label{cor-3-6}
	For every integer $n>1$ we have
	\begin{equation}
		\left(\Lambda_{A}\ast\tau \right)(n)=\frac{1}{2}\tau(n)A(n)
	\end{equation}
\end{corollary}

Substituting $f(n)=A(n)$ into the Ennaoui-Selberg identity (\ref{cor-2-7}) to find that :
\begin{equation}
	A(n)\Lambda_A(n)+\left( \Lambda_A\ast\Lambda_A\right) (n)=\left( \mu\ast A^2\right) (n)
\end{equation}

Now by using the theorem (\ref{the-3-2}) if we take $g(n)=A(n)$ we have that :
\begin{equation}\label{equ-a-ld}
	\left(\Lambda_{Ld}\ast A \right) (n)=A(n)Ld(n)-\frac{1}{2}\left(\Omega_2(n)+\Omega(n) \right) 
\end{equation}
substituting $f=Ld$ and $g=A$ into the corollary (\ref{cor-2-4}) then we have :
\begin{equation}
	\left(1\ast A\Lambda_{Ld} \right)=\left(1\ast Ld\Lambda_{A} \right) (n)=\frac{1}{2}\left(\Omega_2(n)+\Omega(n) \right) 
\end{equation}
\begin{theorem}\label{the-3-4}	for every integer $n>0$ we have :
	\begin{equation}
		\left( \Lambda_{A}\ast\beta_k\right)(n)=A(n)\beta_k(n)-\beta_{k+1}(n)
	\end{equation}
\end{theorem}
\begin{proof}
	for any positive integer $n$ such that  $n = p_1^{\alpha _1}\ldots  p_s^{\alpha _s}$ we have :
	\begin{align*}
		\left( \Lambda_{A}\ast\beta_k\right)(n)&=	\sum _{d\mid n} \Lambda_{A}(d)\beta_k\left(\frac{n}{d} \right)  =\sum \limits_{p^{\alpha}||n}\sum \limits_{i=1}^{i=\alpha} \Lambda_{A}(p^{i})\beta_k\left(\frac{n}{p^{i}} \right) 
		\\
		&=\sum \limits_{p^{\alpha}||n}\sum \limits_{i=1}^{i=\alpha} p\beta_k\left(\frac{n}{p^{i}} \right) 
		\\
		&=\sum \limits_{p^{\alpha}||n}\left( -p^{k+1} +p\sum \limits_{i=1}^{i=\alpha} \beta_k\left(n\right)\right) 
		\\
		&=\sum \limits_{p^{\alpha}||n} -p^{k+1}+\alpha p\beta_k\left(n\right)
		\\
		&=\beta_k(n)\sum \limits_{p^{\alpha}||n}\alpha p-\sum \limits_{p^{\alpha}||n} p^{k+1}
		\\
		&=\beta_k(n)A(n)-\beta_{k+1}(n)
	\end{align*}
\end{proof}

The definition (\ref{def-mangol}) may be considered a generalized von mangoldt function. This terminology arises from the observation that the logarithm is L-additive with $h_{log}(n)=1$; it satisfies the usual von mangoldt function denoted by :
$$
\Lambda(n)=\Lambda_{log} (n)={\begin{cases}\frac{log(p)}{h_{log}(p)}=log(p) &{\text{if }}n=p^{k}{\text{ for some prime }}p{\text{ and integer }}k\geq 1,\\0&{\text{otherwise.}}\end{cases}}
$$
By using the corollary (\ref{cor-2-2}) and (\ref{cor-2-3}) the Properties of the von mangoldt function is hold and we have 
$$
log=1\ast \Lambda \;\;\;and\;\;\; \Lambda=\mu\ast\log=-1\ast\mu\log
$$

	\section{the generalized von mangoldt functions in terms of the Dirichlet serie}

Above we have seen that many fundamental properties of the generalized von mangoldt function. We complete this article by changing our point of view slightly and demonstrate that generalized von mangoldt function can also be studied in terms of the Dirichlet series . 

\medskip

Dirichlet product defined in (\ref{dir-prod}) occurs naturally in the study of Dirichlet series such as the Riemann zeta function. It describes the multiplication of two Dirichlet series in terms of their coefficients: 
\begin{equation}\label{eq:5}
	\bigg(\sum \limits_{n\geq 1}\frac{\big(f*g\big)(n)}{n^s}\bigg)=\bigg(\sum \limits_{n\geq 1}\frac{f(n)}{n^s} \bigg)
	\bigg( \sum \limits_{n\geq 1}\frac{g(n)}{n^s} \bigg)
\end{equation}
with Riemann zeta function or  is defined by : 
$$\zeta(s)= \sum \limits_{n\geq 1} \frac{1}{n^s}$$
These functions are widely studied in the literature (see, e.g., \cite{book1, book2, book3}).\\
For later convenience we introduce the prime zeta function, described in Froberg (1968) (see, e.g., \cite{et-prime}), denoted by $P(s)$. We define it by :
$$P(s)= \sum \limits_{p} \frac{1}{p^s}$$
	In the next of this section we will use this notation :
$$
P_f(s)=\sum \limits_{p}\frac{f(p)}{p^{s}-1}
$$ 
\begin{theorem}	
	For $s\in\mathbb{C}$ such that $Re(s)>1$ we have :
	\begin{equation*}
		\sum \limits_{n\geq 1} \frac{\tau(n)\Omega(n)}{n^s}=2\zeta^2(s)P_{\Omega}(s)
	\end{equation*}
\end{theorem}
\begin{proof}
Let $s\in\mathbb{C}$ such that $Re(s)>1$ , then by using the corollary $\left( \ref{cor-3-2}\right)$  we have : 
	$$
		\left(\Lambda_{\Omega}\ast\tau \right)(n)=\frac{1}{2}\tau(n)\Omega(n)
	$$
then by the formula $\left(\ref{eq:5} \right) $ we have  :
$$
\sum \limits_{n\geq 1} \frac{\tau(n)\Omega(n)}{n^s}=\sum \limits_{n\geq 1} \frac{\left(\Lambda_{\Omega}\ast\tau \right)(n)}{n^s}=2
\left(\sum \limits_{n\geq 1} \frac{\tau(n)}{n^s} \right) \left(\sum \limits_{n\geq 1} \frac{\Lambda_{\Omega}(n)}{n^s} \right) 
$$
Since (see, e.g., \cite{cat-serie}) :
$$
\sum \limits_{n\geq 1} \frac{\tau(n)}{n^s}=\zeta^2(s)
$$
and by the corollary (\ref{cor-2-9}) we have :
$$
P_{\Omega}(s)=\sum \limits_{n\geq 1} \frac{\Lambda_{\Omega}(n)}{n^s}
$$ then we conclude that :
$$
	\sum \limits_{n\geq 1} \frac{\tau(n)\Omega(n)}{n^s}=2\zeta^2(s)P_{\Omega}(s)
$$
\end{proof}
\begin{theorem}	
	For $s\in\mathbb{C}$ such that $Re(s)>max (1, 1 + k)$ we have :
	\begin{equation*}
		\sum \limits_{n\geq 1} \frac{\Omega(n)\beta_k(n)}{n^s}=\zeta(s)P(s-k)\bigg(P_{\Omega}(s)+1 \bigg) 
	\end{equation*}
\end{theorem}
\begin{proof}
	Let $s\in\mathbb{C}$ such that $Re(s)>Max(1,k+1)$ , then By using the theorem $\left( \ref{the-3-1}\right)$  we have : 
	$$
	\left( \Lambda_{\Omega}\ast\beta_k\right)(n)=\Omega(n)\beta_k(n)-\beta_k(n)
	$$
	then by the formula $\left(\ref{eq:5} \right) $ we have  :
	$$
	\sum \limits_{n\geq 1} \frac{\Omega(n)\beta_k(n)}{n^s}=\sum \limits_{n\geq 1} \frac{	\left( \Lambda_{\Omega}\ast\beta_k\right)(n)}{n^s}+\sum \limits_{n\geq 1} \frac{	\beta_k(n)}{n^s}=
	\left(\sum \limits_{n\geq 1} \frac{\beta_k(n)}{n^s} \right) \left(\sum \limits_{n\geq 1} \frac{\Lambda_{\Omega}(n)}{n^s} \right)+\sum \limits_{n\geq 1} \frac{	\beta_k(n)}{n^s}
	$$
	Since (see, e.g., \cite{cat-serie}) :
	$$
	\sum \limits_{n\geq 1} \frac{\beta_k(n)}{n^s}=\zeta(s)P(s-k)
	$$
	and by the corollary (\ref{cor-2-9}) we have :
	$$
	P_{\Omega}(s)=\sum \limits_{n\geq 1} \frac{\Lambda_{\Omega}(n)}{n^s}
	$$ then we conclude that :
	$$
	\sum \limits_{n\geq 1} \frac{\Omega(n)\beta_k}{n^s}=\zeta(s)P(s-k)P_{\Omega}(s)+\zeta(s)P(s-k)
	$$
	which completes the proof
\end{proof}
\begin{theorem}	
	For $s\in\mathbb{C}$ such that $Re(s)>max (1, k-1)$ we have :
	\begin{equation*}
		\sum \limits_{n\geq 1} \frac{Ld(n)\beta_k(n)}{n^s}=\zeta(s)\bigg(P(s-k)P_{Ld}(s)+P(s-k+1) \bigg) 
	\end{equation*}
\end{theorem}
\begin{proof}
	Let $s\in\mathbb{C}$ such that $Re(s)>max (1, k-1)$ , then by  the theorem $\left( \ref{the-3-3}\right)$  we have : 
	$$
		\left(\Lambda_{Ld}\ast\beta_k \right)(n)=Ld(n)\beta_k(n)-\beta_{k-1}(n)
	$$
	then by the formula $\left(\ref{eq:5} \right) $ we have  :
$$
\sum \limits_{n\geq 1} \frac{Ld(n)\beta_k(n)}{n^s}=\sum \limits_{n\geq 1} \frac{	\left( \Lambda_{Ld}\ast\beta_k\right)(n)}{n^s}+\sum \limits_{n\geq 1} \frac{	\beta_{k-1}(n)}{n^s}=
\left(\sum \limits_{n\geq 1} \frac{\beta_k(n)}{n^s} \right) \left(\sum \limits_{n\geq 1} \frac{\Lambda_{Ld}(n)}{n^s} \right)+\sum \limits_{n\geq 1} \frac{	\beta_{k-1}(n)}{n^s}
$$
Since (see, e.g., \cite{cat-serie}) :
$$
\sum \limits_{n\geq 1} \frac{\beta_k(n)}{n^s}=\zeta(s)P(s-k)
$$
and by the corollary (\ref{cor-2-9}) we have :
$$
P_{Ld}(s)=\sum \limits_{n\geq 1} \frac{\Lambda_{Ld}(n)}{n^s}
$$ then we find that :
$$
\sum \limits_{n\geq 1} \frac{Ld(n)\beta_k}{n^s}=\zeta(s)P(s-k)P_{Ld}(s)+\zeta(s)P(s-k+1)
$$
which completes the proof .
\end{proof}
\begin{theorem}	
	For $s\in\mathbb{C}$ such that $Re(s)>max (1, k+2)$ we have :
	\begin{equation*}
		\sum \limits_{n\geq 1} \frac{A(n)\beta_k(n)}{n^s}=\zeta(s)\bigg(P(s-k)P_{Ld}(s)+P(s-k-1) \bigg) 
	\end{equation*}
\end{theorem}
\begin{proof}
	Let $s\in\mathbb{C}$ such that $Re(s)>max (1, k+2)$ , then by  the theorem $\left( \ref{the-3-4}\right)$  we have : 
	$$
	\left(\Lambda_{A}\ast\beta_k \right)(n)=A(n)\beta_k(n)-\beta_{k+1}(n)
	$$
	then by the formula $\left(\ref{eq:5} \right) $ we have  :
	$$
	\sum \limits_{n\geq 1} \frac{A(n)\beta_k(n)}{n^s}=\sum \limits_{n\geq 1} \frac{	\left( \Lambda_{A}\ast\beta_k\right)(n)}{n^s}+\sum \limits_{n\geq 1} \frac{	\beta_{k+1}(n)}{n^s}=
	\left(\sum \limits_{n\geq 1} \frac{\beta_k(n)}{n^s} \right) \left(\sum \limits_{n\geq 1} \frac{\Lambda_{A}(n)}{n^s} \right)+\sum \limits_{n\geq 1} \frac{	\beta_{k+1}(n)}{n^s}
	$$
	Since (see, e.g., \cite{cat-serie}) :
	$$
	\sum \limits_{n\geq 1} \frac{\beta_{k}(n)}{n^s}=\zeta(s)P(s-k)
	$$
	and by the corollary (\ref{cor-2-9}) we have :
	$$
	P_{A}(s)=\sum \limits_{n\geq 1} \frac{\Lambda_{A}(n)}{n^s}
	$$ then we find that :
	$$
	\sum \limits_{n\geq 1} \frac{A(n)\beta_k}{n^s}=\zeta(s)P(s-k)P_{A}(s)+\zeta(s)P(s-k-1)
	$$
	which completes the proof .
\end{proof}
\section{Conclusion :}
	The von Mangoldt function $\Lambda_f$ related to the L-additive function $f$ is another way to be solved many problem of  the Dirichlet series of the  arithmetic function .
		
\end{document}